\documentclass[12pt]{amsart}
\usepackage{amsmath}

\newtheorem{Lemma}             {Lemma}
\newtheorem{Proposition}[Lemma]{Proposition}
\newtheorem{Theorem}    [Lemma]{Theorem}

\newcommand{\Mod}[1]{\ (\mathrm{mod}\ #1)}

\newcommand{\op}{\operatorname}
\newcommand{\sym}{\mathcal S}
\newcommand{\alt}{\mathcal A}

\title{The quadratic type of the 2-principal indecomposable modules of the double covers of alternating groups}
\author{John Murray}
\address{Department of Mathematics \& Statistics, National University of Ireland Maynooth, Co Kildare, Ireland}
\email{John.Murray@mu.ie}
\date{March 9, 2018}
\begin{document}
\maketitle

\section{Introduction}

Our topic is the representation theory of the double cover ${2.\alt_n}$ of the alternating group $\alt_n$ over an algebaically closed field $k$ of characteristic $2$, for $n\geq5$. Recall that an element of a group is said to be $2$-regular if it has odd order, real if it is conjugate to its inverse, and strongly real if some involution conjugates it to its inverse. The following is a recent result of R. Gow and the author \cite{GowM}:

\begin{Proposition}\label{P:GM}
Let $k$ be an algebraically closed field of characteristic $2$ and let $G$ be a finite group. Then the number of isomorphism classes of principal indecomposable $kG$-modules which have a quadratic geometry is equal to the number of strongly real $2$-regular conjugacy classes of $G$.
\end{Proposition}

Now all real elements of $\alt_n$ are strongly real. So every self-dual principal indecomposable $k\alt_n$-module has quadratic type. On the other hand, there may be real elements of $2.\alt_n$ which are not strongly real. The purpose of this note is to determine the type of each principal indecomposable $k(2.\alt_n)$-module.

Let $\sym_n$ be the symmetric group of degree $n$ and let ${\mathcal D}(n)$ be the set of partitions of $n$ which have distinct parts. In \cite[11.5]{James} G. James constructed an irreducible $k\sym_n$-module $D^\mu$ for each partition $\mu\in{\mathcal D}(n)$. Moreover, he showed that the $D^\mu$ are pairwise non-isomorphic, and every irreducible $k\sym_n$-modules is isomorphic to some $D^\mu$. As $\alt_n$ has index $2$ in $\sym_n$, Clifford theory shows that the restriction $D^\mu{\downarrow_{\alt_n}}$ is either irreducible or splits into a direct sum of two non-isomorphic irreducible $k\alt_n$-modules. Moreover, every irreducible $k\alt_n$-module is a direct summand of some $D^\mu{\downarrow_{\alt_n}}$.

D. Benson determined \cite{Benson} which $D^\mu{\downarrow_{\alt_n}}$ are reducible and we recently determined \cite{Murray17} when the irreducible direct summands of $D^\mu{\downarrow_{\alt_n}}$ are self-dual (see below for details). We will use the notation $D_A^\mu$ to denote an irreducible direct summand of $D^\mu{\downarrow_{\alt_n}}$.

As the centre of $2.\alt_n$ acts trivially on any irreducible module, the groups $\alt_n$ and $2.\alt_n$ have the same irreducible modules over $k$. We will use $P^\mu$ to denote the projective cover of $D^\mu_A$ as $k(2.\alt_n)$-module. But note that $P^\mu$ is not an $k\alt_n$-module.

Let $|\mu|_a=\sum(-1)^{j+1}\mu_j$ be the alternating sum of a partition $\mu$ and let $\ell_o(\mu)$ be the number of odd parts in $\mu$. Then $|\mu|_a\equiv \ell_o(\mu)\equiv n\Mod2$ and $|\mu|_a\leq \ell_o(\mu)$, if $\mu$ has distinct parts. Our result is:

\begin{Theorem}\label{T:main}
Let $\mu$ be a partition of\/ $n$ which has distinct parts and let $P^\mu$ be the principal indecomposable $k(2.\alt_n)$-module corresponding to the simple $k(2.\alt_n)$-module $D_A^\mu$. Then $P^\mu$ has quadratic type if and only if there is an integer $m$ such that $\frac{n-|\mu|_a}{2}\leq4m\leq\frac{n-\ell_o(\mu)}{2}$.
\end{Theorem}

Throughout the paper all our modules are left modules.

\section{Notation and known results}

\subsection{Principal indecomposable modules}

Let $G$ be a finite group and let $k$ be an algebraically closed field of characteristic $p>0$. The group algebra $kG$ is a $k$-algebra together with a group homomorphism $G\rightarrow\op{U}(kG)$ whose image is a basis for $kG$. The ring multiplication makes $kG$ into a left module over itself, the so-called regular $kG$-module. The indecomposable summands of the regular module are called the principal indecomposable $kG$-modules.

Now $\op{End}_{kG}(kG)\cong kG^{op}$ as $k$-algebras; $x\in kG$ corresponds to the map $r_x:kG\rightarrow kG$, $r_x(y)=yx$, for all $y\in kG$. It follows from this that if $P$ is a direct summand of $kG$, then there is a primitive idempotent $e\in kG$ such that $P=kGe$. Conversely, given a primitive idempotent $e\in kG$, we have $kG=kGe\oplus kG(1-e)$, and $kGe$ is a principal indecomposable $kG$-module.

Let $P$ be a principal indecomposable $kG$-module. Then $P$ has a unique maximal submodule $\op{rad}(P)$ and a unique irreducible submodule $\op{soc}(P)$. It is known that $P/\op{rad}(P)\cong\op{soc}(P)$. Moreover this irreducible module determines the isomorphism type of $P$. Conversely, given an irreducible $kG$-module $S$, there is a surjection $\phi:kG\rightarrow S$. Then $\phi$ splits as $kG$ is projective. So $kG=\op{ker}(\phi)\oplus P$, for some projective $kG$-module $P$. In particular $P/\op{rad}(P)\cong S\cong\op{soc}(P)$. In this way there is a canonical one-to-one correspondence between the isomorphism classes of principal indecomposable $kG$-modules and the isomorphism classes of irreducible $kG$-modules.

Recall that a conjugacy class of $G$ is said to be $p$-regular if the elements of the class have order coprime to $p$. It is a theorem of R. Brauer that the number of irreducible $kG$-modules equals the number of $p$-regular conjugacy classes of $G$. Thus the number of isomorphism classes of principal indecomposable $kG$-modules equals the number of $p$-regular conjugacy classes of $G$.

We say that $g\in G$ is real in $G$ if there exists $x\in G$ such that $g^{-1}=xgx^{-1}$, strongly real if $x$ can be chosen to be an involution, and weakly real if it is real but not strongly real.

Now recall that the dual of a $kG$-module $M$ is the (left) $kG$-module $M^*=\op{Hom}_k(M,k)$. Moreover $M$ is self-dual (isomorphic to $M^*$) if and only if $M$ affords a non-degenerate $G$-invariant $k$-valued bilinear form. Clearly $P$ is self-dual if and only if $\op{soc}(P)$ is self-dual. Moreover, the number of isomorphism classes of self-dual irreducible $kG$-modules equals the number of real $p$-regular conjugacy classes of $G$.

\subsection{Symplectic and quadratic forms}

Let $(K,R,k)$ be a $p$-modular system for $G$, where for the moment $p$ is an arbitrary prime integer. So $R$ is discrete valuation ring of characteristic $0$, with unique maximal ideal $J$ containing $p$. Moreover $K$ is the field of fractions of $R$, the residue field $k=R/J$ has characteristic $p$ and $R$ is complete with respect to the topology induced by the valuation. We also assume that $K$ and $k$ are splitting fields for all subgroups of $G$.

In this context every principal indecomposable $kG$-module $P$ has a unique lift to a principal indecomposable $RG$-module $\hat P$ (this means that $\hat P$ is a finitely generated free $RG$-module, which is projective as $RG$-module, and the $kG$-module $\hat P/J\hat P$ is isomorphic to $P$).

We say that a $kG$-module $M$ has quadratic, orthogonal or symplectic type if $M$ affords a non-degenerate $G$-invariant quadratic form, symmetric bilinear form or symplectic bilinear form, respectively.

By a result of R. Gow \cite[8.11]{BH}, if $p$ is odd then $P$ is self-dual if and only if it has orthogonal or symplectic type, and these possibilities are mutually exclusive. Moreover the type of $P$ coincides with the type of $\op{soc}(P)$. Note that quadratic type is the same as orthogonal type, as $p\ne2$.

Suppose now that $p=2$. P. Fong showed that each self-dual non-trivial irreducible $kG$-module has symplectic type. This form is unique up to scalars, by Schur's Lemma. Gow and Willems showed that a principal indecomposable $kG$-module has quadratic type if and only if it has symplectic type. However there may be self-dual principal indecomposable $kG$-modules which are not of quadratic type. In particular, there is no relation between the type of $P$ and the type of $\op{soc}(P)$. The following is proved in \cite{GowM}:

\begin{Proposition}\label{P:strong+weak}
Let $P$ be a self-dual principal indecomposable $kG$-module, let $\Phi$ be the character of\/ $\hat P$ and let $\varphi$ be the Brauer character of\/ $\op{soc}(P)$. Then the following are equivalent:
\begin{itemize}
\item[(i)] $P$ has quadratic type.
\item[(ii)] There is an involution ${t\in G}$ and ${x\in\op{soc}(P)}$ such that\\ $B(tx,x)\ne0$; here $B$ is a symplectic form on $\op{soc}(P)$.
\item[(iii)] There is an involution $t\in G$ and a primitive idempotent $e\in kG$, such that $t^{-1}et=e^o$ and $P\cong kGe$.
\item[(iv)] $\hat P$ has quadratic type.
\item[(v)] $\varphi(g)\not\in2R$, for some strongly real $2$-regular $g\in G$.
\item [(vi)] $\frac{\Phi(g)}{|C_G(g)|}\in2R$, for all weakly real $2$-regular $g\in G$.
\end{itemize}
\end{Proposition}

We will use the equivalence (i)$\Longleftrightarrow$(ii) to prove Theorem \ref{T:main}.

\section{The double covers of alternating groups}

\subsection{Strongly real classes}

For $n\geq2$, the alternating group $\alt_n$ is the subgroup of even permutations in the symmetric group $\sym_n$. For $n\geq4$, $\alt_n$ has a double cover group $2.\alt_n$, which is unique up to isomorphism, and which is a subgroup of either of the (generally) two isoclinic double covers $2.\sym_n$ of $\sym_n$. Now $\{\alt_n\mid n\geq5\}$ is an infinite family of finite simple groups. So $2.\alt_n$ is a perfect group, for $n\geq5$, and is known to be a Schur covering group of $\alt_n$ if  $n\ge5$, $n\ne6,7$.

Given distinct $i_1,i_2,\dots,i_m\in\{1,\dots,n\}$, we use $(i_1,i_2,\dots,i_m)$ to denote an $m$-cycle in $\sym_n$. So $(i_1,i_2,\dots,i_m)$ maps $i_j$ to $i_{j+1}$, for $j=1,\dots,m-1$, sends $i_m$ to $i_1$ and fixes all $i\ne i_1,\dots,i_m$. Now each permutation $\sigma\in\sym_n$ has a unique factorization as a product of disjoint cycles. If we arrange the lengths of these cycles in a non-increasing sequence, we get a partition of $n$, which is called the cycle type of $\sigma$. The set of permutations with a fixed cycle type $\lambda$ is a conjugacy class of $\sym_n$, here denoted $C_\lambda$. In particular the $2$-regular conjugacy classes of $\sym_n$ are indexed by the set ${\mathcal O}(n)$ of partitions of $n$ whose parts are odd.

A transposition in $\sym_n$ is a $2$-cycle $(i,j)$ where $i,j$ are distinct elements of $\{1,\dots,n\}$. So $(i,j)$ has cycle type $(2,1^{n-2})$. It is clear that there is one conjugacy class of involutions for each partition $(2^m,1^{n-2m})$ of $n$, with $1\leq m\leq n/2$. We call a product of $m$-disjoint transpositions an $m$-involution in $\sym_n$. It follows that $\sym_n$ has $\lfloor\frac{n}{2}\rfloor$ conjugacy classes of involutions; the $m$-involutions, for $1\leq m\leq n/2$.

Suppose that $\pi=(i_1,i_{1+m})(i_2i_{2+m})\dots(i_m,i_{2m})$ is an $m$-involution in $\sym_n$. Then we say that $(i_1,i_{1+m}),(i_2i_{2+m}),\dots,(i_m,i_{2m})$ are the transpositions in $\pi$ and write $(i_j,i_{j+m})\in \pi$, for $j=1,\dots,m$. Notice that each $\{i_j,i_{j+m}\}$ is a non-singleton orbit of $\pi$ on $\{1,\dots,n\}$.

Given a partition $\lambda$ of $n$, we have $C_\lambda\subseteq\alt_n$ if and only if $\lambda$ is even ($n\equiv\ell(\lambda)\mod2$). Moreover for $\lambda$ even, $C_\lambda$ is a union of two conjugacy classes of $\alt_n$ if $\lambda$ has distinct odd parts and otherwise $C_\lambda$ is a single conjugacy class of $\alt_n$. In either case we use $C_{\lambda,A}$ to denote an $\alt_n$-conjugacy class contained in $C_\lambda$. If $\lambda$ has distinct odd parts then $C_{\lambda,A}$ is a real conjugacy class of $\alt_n$ if and only if $n\equiv\ell(\lambda)\mod4$.

Next let $z\in2.\alt_n$ be the involution which generates the centre of $2.\alt_n$. As $\langle z\rangle$ is a central $2$-subgroup of $2.\alt_n$, there is a one-to-one correspondence between the $2$-regular conjugacy classes of $2.\alt_n$ and the $2$-regular conjugacy classes of $\alt_n\cong(2.\alt_n)/\langle z\rangle$; if $\lambda$ is an odd partition of $n$ the preimage of $C_{\lambda,A}$ in $2.\alt_n$ consists of a single class $\hat C_{\lambda,A}$ of odd order elements and another class $z\hat C_{\lambda,A}$ of elements whose $2$-parts equal $z$.

Notice that an $m$-involution belongs to $\alt_n$ if and only if $m$ is even. Moreover, the $2m$-involutions form a single conjugacy class of $\alt_n$. So $\alt_n$ has $\lfloor\frac{n}{4}\rfloor$ conjugacy classes of involutions; the $2m$-involutions, for $1\leq m\leq n/4$. Now each $2m$-involution in $\alt_n$ is the image of two involutions in $2.\alt_n$, if $m$ is even, or is the image of two elements of order $4$ in $2.\alt_n$, if $m$ is odd.

Set $m_o(\lambda)$ as the number of parts which occur with odd multiplicity in $\lambda$.

\begin{Lemma}\label{L:strongly_real}
If $\lambda$ is a partition of $n$ with all parts odd then $\hat C_{\lambda,A}$ is a strongly real conjugacy class of\/ $2.\alt_n$ if and only if there is an integer $m$ such that $\frac{n-\ell(\lambda)}{2}\leq4m\leq\frac{n-m_o(\lambda)}{2}$.
\end{Lemma}
\begin{proof}
Let $\sigma\in\alt_n$ have cycle type $\lambda$ and let $\pi$ be an $m$-involution in $\sym_n$ which inverts $\sigma$. Set $\ell:=\ell(\lambda)$, and let $X_1,\dots,X_\ell$ be the orbits of $\sigma$ on $\{1,\dots,n\}$. Then $\pi$ permutes the sets $X_1,\dots,X_\ell$.

Let $j>0$. If $\pi X_j=X_j$, then $\pi$ fixes a unique element of $X_j$, and hence acts as an $\frac{|X_j|-1}{2}$-involution on $X_j$. If instead $\pi X_j\ne X_j$, then $\pi$ is a bijection $X_j\rightarrow\pi X_j$. So $\pi$ acts as an $|X_j|$-involution on $X_j\cup\pi X_j$. We may order the $X_j$ and choose $k\geq0$ such that $\pi X_j=X_{j+k}$, for $j=1,\dots,k$, and $\pi X_j=X_j$, for $j=2k+1,2k+2,\dots,\ell$. Then from above
$$
m=\sum_{j=1}^k\frac{|X_j|+|X_{j+k}|}{2}+\sum_{j=2k+1}^\ell\frac{|X_j|-1}{2}=\frac{n+2k-\ell}{2}.
$$
Now the maximum value of $2k$ is $2k=\ell-m_o(\lambda)$, when $\pi$ pairs the maximum number of orbit of $\sigma$ which have equal size. This implies that $m\leq\frac{n-m_o(\lambda)}{2}$. The minimum value of $2k$ is $0$. This occurs when $\pi$ fixes each orbit of $\sigma$. It follows from this that $m\geq\frac{n-\ell(\lambda)}{2}$.

Conversely, it is clear that for each $m>0$ with $\frac{n-\ell}{2}\leq m\leq\frac{n-m_o(\lambda)}{2}$, there is an $m$-involution $\pi\in\sym_n$ which inverts $\sigma$; $\pi$ pairs $\ell+2m-n$ orbits of $\sigma$ and fixes the remaining $n-2m$ orbits of $\sigma$. The conclusion of the Lemma now follows from our description of the involutions in $2.\alt_n$.
\end{proof}

\subsection{The irreducible modules}

By an $n$-tabloid we mean an indexed collection $R=(R_1,\dots R_\ell)$ of non-empty subsets of $\{1,\dots,n\}$ which are pairwise disjoint and whose union is $\{1,\dots,n\}$ (also known as an ordered partition of $\{1,\dots,n\}$). We shall refer to $R_1,\dots,R_\ell$ as the rows of $R$. Set $\lambda_i:=|R_i|$. Then we may choose indexing so that $\lambda=(\lambda_1,\dots,\lambda_\ell)$ is a partition of $n$, which we call the type of $R$. Now $\sym_n$ acts on all $\lambda$-tabloids; the corresponding permutation module (over ${\mathbb Z}$) is denoted $M^\lambda$.

Next recall that the Young diagram of $\lambda$ is a collection of boxes in the plane, oriented in the anglo-american tradition: the first row consists of $\lambda_1$ boxes. Then for $i=2,\dots,\ell$ in turn, the $i$-th row consists of $\lambda_i$ boxes placed directly below the $(i-1)$-th row, with the leftmost box in row $i$ directly below the leftmost box in row $i-1$.

By a $\lambda$-tableau we shall mean a bijection $t:[\lambda]\rightarrow\{1,\dots,n\}$, or a filling of the boxes in the Young diagram with the symbols $1,\dots,n$. So for $1\leq r\leq\ell$ and $1\leq c\leq\lambda_r$, we use $t(r,c)$ to denote the image of the position $(r,c)\in[\lambda]$ in $\{1,\dots,n\}$. Conversely, given $i\in\{1,\dots,n\}$, there is a unique $r=r_t(i)$ and $c=c_t(i)$ such that $t(r,c)=i$. We say that $i$ is in row $r$ and column $c$ of $t$.

Clearly there are $n!$ tableau of type $\lambda$ and $\sym_n$ acts regularly on the set of $\lambda$-tableau. For $\sigma\in\sym_n$, we define $\sigma t:[\lambda]\rightarrow\{1,\dots,n\}$ as the composition $(\sigma t)(r,c)=\sigma(t(r,c))$, for all $(r,c)\in[\lambda]$. In other words, the permutation module of $\sym_n$ acting on tableau is (non-canonically) isomorphic to the regular module ${\mathbb Z}\sym_n$; once we fix $t$, we may identify $\sigma\in\sym_n$ with the tableau $\sigma t$.

Associated with $t$, we have two important subgroups of $\sym_n$. The column stabilizer of $t$ is $C_t:=\{\sigma\in\sym_n\mid c_t(i)=c_t(\sigma i)$, for $i=1,\dots,n\}$ and 
the row stabilizer of $t$ is $R_t:=\{\sigma\in\sym_n\mid r_t(i)=r_t(\sigma i)$, for $i=1,\dots,n$.

We use $\{t\}$ to denote the tabloid formed by the rows of $t$. So $\{t\}_r:=\{t(r,c)\mid1\leq c\leq\lambda_r\}$, for $r=1,\dots,\ell$. Also $\{s\}=\{t\}$ if and only if $s=\sigma t$, for some $\sigma\in R_t$.
Notice that the actions of $\sym_n$ on tableau and tabloids are compatible, in the sense that $\sigma\{t\}=\{\sigma t\}$. In other words, the map $t\mapsto\{t\}$ induces a surjective $\sym_n$-homomorphism ${\mathbb Z}\sym_n\rightarrow M^\lambda$. The kernel of this homomorphism is the ${\mathbb Z}$-span of $\{\sigma t\mid \sigma\in R_t\}$.

The polytabloid $e_t$ associated with $t$ is the following element of $M^\lambda$:
$$
e_t:=\sum_{\sigma\in C_t}\op{sgn}(\sigma)\{\sigma t\}.
$$
We use $\op{supp}(t):=\{\{\sigma t\}\mid \sigma\in C_t\}$ to denote the set of tabloids which occur in the definition of $e_t$. Note that $e_{\pi t}=\op{sgn}(\pi)e_t$, for all $\pi\in C_t$. In particular, if $r_t(i)=r_t(j)$, then $e_{(i,j)t}=-e_t$. Also if $\pi\in\sym_n$, then $C_{\pi t}=\pi C_t\pi^{-1}$ and $R_{\pi t}=\pi R_t\pi^{-1}$. So $e_{\pi t}=\pi e_t$ and $\op{supp}(\pi t)=\pi\op{supp}(t)$.

The ${\mathbb Z}$-span of all $\lambda$-polytabloids is a $\sym_n$-submodule of $M^\lambda$ called the Specht module. It is denoted by $S^\lambda$. So $S^\lambda$ is a finitely generated free ${\mathbb Z}$-module (${\mathbb Z}$-lattice).

\subsection{Involutions and bilinear forms}

Let $\mu=(\mu_1>\dots>\mu_{2s-1}>\mu_{2s}\geq0)$ be a partition of $n$ which has distinct parts. Set $\langle\,,\,\rangle$ as the symmetric bilinear form on $M^\mu$ with respect to which the $\mu$-tabloids form an orthonormal basis. Then $\langle\,,\,\rangle$ restricts to a bilinear form on $S^\mu$. Set $(S^\mu)^\perp:=\{m\in M^\mu\mid\langle m,s\rangle\in 2{\mathbb Z},\forall s\in S^\mu\}$. According to James, $D^\mu:=S^\mu/S^\mu\cap(S^\mu)^\perp$ is a non-zero irreducible $k\sym_n$-module.

Benson's classification of the irreducible $k\alt_n$-modules \cite{Benson}, and our classification of the self-dual irreducible $k\alt_n$-modules \cite{Murray17}, as mentioned in the introduction, is given by:

\begin{Lemma}\label{L:Benson_Murray}
${D^\mu\downarrow_{\alt_n}}$ is reducible if and only if for each $j>0$
$$
\mbox{(i) }\,\mu_{2j-1}-\mu_{2j}=\mbox{$1$ or $2$\qquad and}\quad\mbox{(ii) }\,\mu_{2j-1}+\mu_{2j}\not\equiv2\Mod4.
$$
When $D^\mu\downarrow_{\alt_n}$ is reducible, its irreducible direct summands are self-dual if and only if\/ $\sum_{j>0}\mu_{2j}$ is even.
\end{Lemma}

We note the following:
 
\begin{Lemma}
Suppose that $\mu\ne(n)$ and let $\phi:S^\mu\rightarrow D^\mu$ be the ${\mathbb Z}\sym_n$-projection. Set $B(\pi x,\pi y):=\langle x,y\rangle$, for $x,y\in S^\mu$. Then $B$ is a well-defined non-zero symplectic bilinear form on $D^\mu$.
\end{Lemma}

Notice that if $x,y\in D^\mu$ and $\pi$ is an involution in $\sym_n$ then
$$
B(\pi(x+y),x+y)=B(\pi x,x)+B(\pi y,y).
$$
So we can focus on a single polytabloid in $S^\lambda$: 

\begin{Lemma}\label{L:supp}
If\/ $t$ is a $\mu$-tableau and $\pi$ is an involution in $\sym_n$, then
$$
\langle \pi e_t,e_t\rangle\equiv|\{T\in\op{supp}(\pi t)\cap\op{supp}(t)\mid \pi T=T\}|\mod2.
$$
In particular, if\/ $\langle \pi e_t,e_t\rangle$ is odd, then $\pi\in R_{\sigma t}$, for some $\sigma\in C_t$.
\begin{proof}
We have
$$
\begin{aligned}
\langle \pi e_t,e_t\rangle
&=\sum_{\sigma_1,\sigma_2\in C_t}\op{sgn}(\pi\sigma_1\pi^{-1})\op{sgn}(\sigma_2)\langle(\pi\sigma_1\{t\},\sigma_2\{t\}\rangle\\
&\equiv|\{ (\sigma_1,\sigma_2)\in C_t\times C_t\mid\pi\sigma_1\{t\}=\sigma_2\{t\} \}|\Mod2\\
&=|\op{supp}(\pi t)\cap\op{supp}(t)|.
\end{aligned}
$$
Now notice that $T\mapsto\pi T$ is an involution on $\op{supp}(\pi t)\cap\op{supp}(t)$. So $|\op{supp}(\pi t)\cap\op{supp}(t)|\equiv|\{T\in\op{supp}(\pi t)\cap\op{supp}(t)\mid \pi T=T\}|\mod2$.

Suppose that $\langle \pi e_t,e_t\rangle$ is odd. Then by the above, there exists $\sigma\in C_t$ such that $\pi\{\sigma t\}=\{\sigma t\}$. This means that $\pi\in R_{\sigma t}$. 
\end{proof}
\end{Lemma}

\begin{Lemma}\label{L:upper}
Let $t$ be a $\mu$-tableau and let $m$ be a positive integer such that $\langle \pi e_t,e_t\rangle$ is odd, for some $m$-involution $\pi\in\sym_n$. Then $m\leq\frac{n-\ell_o(\mu)}{2}$ and $\pi$ fixes at most one entry in each column of\/ $t$.
\begin{proof}
By the previous Lemma, we may assume that $\pi\in R_t$. Now $R_t\cong\sym_\mu$. For $i>0$, there is $j$-involution in $\sym_i$ for $j=1,\dots\lfloor\frac{i}{2}\rfloor$. So there is an $m$-involution in $R_t$ if and only if
$$
m\leq\sum\left\lfloor\frac{\mu_i}{2}\right\rfloor=\sum_{\mu_i\mbox{ even}}\frac{\mu_i}{2}+\sum_{\mu_i\mbox{ odd}}\frac{\mu_i-1}{2}=\frac{n-\ell_o(\mu)}{2}
$$

Let $i,j$ belong to a single column of $t$. We claim that $i,j$ belong to different columns of $\pi t$. For suppose otherwise. Then $(i,j)\in C_t\cap C_{\pi t}$. So the map $T\mapsto(i,j)T$ is an involution on $\op{supp}(\pi t)\cap\op{supp}(t)$ which has no fixed-points. In particular $|\op{supp}(\pi t)\cap\op{supp}(t)|$ is even, contrary to hypothesis. This proves the last assertion.
\end{proof}
\end{Lemma}

We can now prove a key technical result:

\begin{Lemma}\label{L:lower}
Let $t$ be a $\mu$-tableau and let $m$ be a positive integer such that $\langle \pi e_t,e_t\rangle$ is odd, for some $m$-involution $\pi\in\sym_n$. Then $m\geq\frac{n-|\mu|_a}{2}$.
\begin{proof}
Let $T\in\op{supp}(\pi t)\cap\op{supp}(t)$ such that $\pi T=T$. Write $\pi_j$ for the restriction of $\pi$ to the rows $T_{2j-1}$ and $T_{2j}$ of $T$, for each $j>0$. Then there is $m_j\geq0$ such that $\pi_j$ is an $m_j$-involution, for each $j>0$. So $m=\sum m_j$ and $\pi=\pi_1\pi_2\dots\pi_{\lfloor\frac{\ell(\mu)+1}{2}\rfloor}$.

We assume for the sake of contradiction that $m<\frac{n-|\mu|_a}{2}$. Now $\frac{n-|\mu|_a}{2}=\sum_{j>0}\mu_{2j}$. So $m_j<\mu_{2j}$ for some $j>0$, and we choose $j$ to be the smallest such positive integer.

There is a unique $\sigma\in C_t$ such that $T=\{\sigma t\}$. Set $s=\sigma t$. So $\pi\in R_s$.  We define the graph $\op{Gr}_\pi(s)$ of $\pi$ on $s$ as follows. The vertices of $\op{Gr}_\pi(s)$ are labels $1,\dots,\mu_{2j-1}$ of the columns which meet row $\mu_{2j-1}$ of $s$. There is an edge $c_1\longleftrightarrow c_2$ if and only if one of the two transpositions $(s(2j-1,c_1),s(2j-1,c_2))$ or $(s(2j,c_1),s(2j,c_2))$ belongs to $\pi_j$. As there are at most two entries in each column of $s$ which are moved by $\pi_j$, it follows that each connected component of $\op{Gr}_\pi(s)$ is either a line segment or a simple closed curve.

We claim that $\op{Gr}_\pi(s)$ has a component with a vertex set contained in $\{1,\dots,\mu_{2j}\}$. For otherwise every component $\Gamma$ of $\op{Gr}_\pi(s)$ is a line segment and $|\op{Edge}(\Gamma)|\geq|\op{Vx}(\Gamma)\cap\{1,\dots,\mu_{2j}\}|$. Summing over all $\Gamma$ we get the contradiction
$$
\mu_{2j}=\sum_{\Gamma}|\op{Vx}(\Gamma)\cap\{1,\dots,\mu_{2j}\}|\leq\sum_{\Gamma}|\op{Edge}(\Gamma)|=m_j.
$$

Now let $X$ be the union of the component of $\op{Gr}_\pi(s)$ which are contained in $\{1,\dots,\mu_{2j}\}$ and let $\Gamma$ be the component of $\op{Gr}_\pi(s)$ which contains $\op{min}(X)$. In particular $\op{Vx}(\Gamma)\subseteq\{1,\dots,\mu_{2j}\}$.

Consider the involution $\sigma_\Gamma:=\prod_{c\in\op{Vx}(\Gamma)}(t(2j-1,c),t(2j,c))$. This transposes the entries between rows $2j-1$ and $2j$ in each column in $\op{Vx}(\Gamma)$. Now it is clear that $\pi$ is in the row stabilizer of $\sigma_\Gamma s$. So $\{\sigma_\Gamma s\}\in\op{supp}(\pi t)\cap\op{supp}(t)$. Moreover, $\op{Gr}_\pi(s)=\op{Gr}_\pi(\sigma_\Gamma s)$ and $s=\sigma_\Gamma(\sigma_\Gamma s)$. It follows that the pair $T\ne\sigma_\Gamma T$ of tabloids makes zero contribution to $\langle \pi e_t,e_t\rangle$ modulo $2$. But $T$ is an arbitrary $\pi$-fixed tabloid in $\op{supp}(\pi t)\cap\op{supp}(t)$. So $\langle\pi e_t,e_t\rangle$ is even, according to Lemma \ref{L:supp}. This contradiction completes the proof.
\end{proof}
\end{Lemma}

\subsection{Proof of Theorem \ref{T:main}}

\begin{proof}[\unskip\nopunct]
\hspace{1em} Suppose first that $P^\mu$ has quadratic type. Then by part (vi) of Proposition \ref{P:strong+weak}, there is an involution $\hat\pi\in2.\alt_n$ such that $B(\hat\pi x,x\rangle$ is odd, for some $x\in D_A^{\mu}$. Let $\pi$ be the image of $\hat\pi$ in $\alt_n$. Then there exists a $\mu$-tableau $t$ such that $\langle\pi e_t,e_t\rangle$ is odd. Now $\pi$ is a $4m$-involution, for some $m>0$. Lemmas \ref{L:upper} and \ref{L:lower} imply that 
$\frac{n-|\mu|_a}{2}\leq 4m\leq\frac{n-\ell_o(\mu)}{2}$. This proves the `only if' part of the Theorem.

According to Lemma \ref{L:strongly_real}, the strongly real $2$-regular conjugacy classes of $2.\alt_n$ are enumerated by $\lambda\in{\mathcal O}(n)$ such that there is a positive integer $m$ with $\frac{n-\ell(\lambda)}{2}\leq 4m\leq\frac{n-m_o(\lambda)}{2}$ (if $\lambda$ has distinct parts, $\frac{n-\ell(\lambda)}{2}=\frac{n-m_o(\lambda)}{2}$ and there are two $2$-regular classes of $2.\alt_n$ labelled by $\lambda$. In all other cases there is a single $2$-regular class of $2.\alt_n$ labelled by $\lambda$.).

Now by Theorem 2.1 in \cite{CGJL} (or the main result in \cite{Murray18}) there exists a bijection $\phi:{\mathcal O}(n)\rightarrow{\mathcal D}(n)$ such that $\ell(\lambda)=|\phi(\lambda)|_a$ and $m_o(\lambda)=\ell_o(\phi(\lambda))$, for all $\lambda\in{\mathcal O}(n)$. Then from the previous paragraph the number of strongly real $2$-regular conjugacy classes of $2.\alt_n$ coincides with the number of irreducible $k(2.\alt_n)$-modules enumerated by $\mu\in{\mathcal D}(n)$ such that $\frac{n-|\mu|_a}{2}\leq 4m\leq\frac{n-\ell_o(\mu)}{2}$ for some integer $m$. However, from earlier in the proof, these are the only $P^\mu$ which can be of quadratic type. We conclude from Proposition \ref{P:GM} that each of these $P^\mu$ is in fact of quadratic type, and no other $P^\mu$ can be of quadratic type. 
\end{proof}

\subsection{Example with $2.\alt_{13}$}

We determine the quadratic type of all projective indecomposable $k(2.\alt_{13})$-modules. Here are the $18$ partitions in ${\mathcal D}(13)$:
$$\begin{array}{|l|l|l|l|}\hline
\mu&\frac{n-|\mu|_a}{2}&\frac{n-\ell_o(\mu)}{2}&\mbox{type}\\
\hline
(7,6)&6&6&\mbox{2 non-quadratic}\\
(8,5)&5&6&\mbox{non-quadratic}\\
(6,5,2)&5&6&\mbox{non-quadratic}\\
(6,4,2,1)&5&6&\mbox{non-quadratic}\\
(5,4,3,1)&5&5&\mbox{2 not self-dual}\\
(7,5,1)&5&5&\mbox{2 not self-dual}\\
(9,4)&4&6&\mbox{quadratic}\\
(7,4,2)&4&6&\mbox{quadratic}\\
(6,4,3)&4&6&\mbox{quadratic}\\
(8,4,1)&4&6&\mbox{quadratic}\\
(7,3,2,1)&4&5&\mbox{quadratic}\\
(10,3)&3&6&\mbox{quadratic}\\
(8,3,2)&3&6&\mbox{quadratic}\\
(9,3,1)&3&5&\mbox{quadratic}\\
(11,2)&2&6&\mbox{quadratic}\\
(10,2,1)&2&6&\mbox{quadratic}\\
(12,1)&1&6&\mbox{quadratic}\\
(13)&0&6&\mbox{quadratic}\\
\hline
\end{array}
$$
Using (i) and (ii) in Lemma \ref{L:Benson_Murray}, we see that $D^{\mu}{\downarrow_{\alt_{13}}}$ is a sum of two non-isomorphic irreducible $k(2.\alt_{13})$-modules for $\mu=(7,6),(5,4,3,1)$ or $(7,5,1)$. For all other $\mu$, $D^{\mu}{\downarrow_{\alt_{13}}}$ is irreducible. So there are $21=18+3$ projective indecomposable $k(2.\alt_{13})$-modules.

By the last statement in Lemma \ref{L:Benson_Murray}, the two irreducible $k(2.\alt_{13})$-modules $D_A^{(5,4,3,1)}$ are duals of each other, as are the two irreducible $k(2.\alt_{13})$-modules $D_A^{(7,5,1)}$. By the same result both irreducible $k(2.\alt_{13})$-modules $D_A^{(7,6)}$ are self-dual. However $6\equiv2\Mod4$. So neither principal indecomposable $k(2.\alt_{13})$-module $P^{(7,6)}$ is of quadratic type.

Next if $\mu=(8,5),(6,5,2)$ or $(6,4,2,1)$ we have $\frac{n-|\mu|_a}{2}=5$ and $\frac{n-\ell_o(\mu)}{2}=6$. So the principal indecomposable $k(2.\alt_{13})$-module $P^\mu$ is not of quadratic type for any of these $\mu$'s. For each of the remaining partitions $\mu$, the principal indecomposable $k(2.\alt_{13})$-module $P^\mu$ is of quadratic type, according to Theorem \ref{T:main}.

\end{document}